\documentclass[10pt]{article}

\usepackage{amsmath}
\usepackage{amssymb}
\usepackage{amsthm}
\usepackage{graphics}
\usepackage[latin1]{inputenc}
\usepackage[T1]{fontenc}

\usepackage{enumerate}

\newtheorem{thm}{Theorem}[section]
\newtheorem{prop}{Proposition}[section]

\newtheorem{rem}{Remark}[section]

\newcommand{\R}{\mathbb{R}}             

\newcommand{\half}{\frac{1}{2}}

\newcommand{\Section}[1]{\section{#1} \setcounter{equation}{0}}

\begin{document}

\title{On non-uniqueness for the anisotropic Calder\'on problem with partial data}
\author{Thierry Daud\'e \footnote{Research supported by the JCJC French National Research Projects Horizons, No. ANR- ANR-16-CE40-0012-01} $^{\,1}$,
Niky Kamran \footnote{Research supported by NSERC grant RGPIN 105490-2011} $^{\,2}$ and Francois Nicoleau \footnote{Research supported by the French National Research Project GDR Dynqua} $^{\,3}$\\[12pt]
 $^1$  \small D\'epartement de Math\'ematiques. UMR CNRS 8088, Universit\'e de Cergy-Pontoise, \\
 \small 95302 Cergy-Pontoise, France. \\
\small Email: thierry.daude@u-cergy.fr \\
$^2$ \small Department of Mathematics and Statistics, McGill University,\\ \small  Montreal, QC, H3A 2K6, Canada. \\
\small Email: nkamran@math.mcgill.ca \\
$^3$  \small  Laboratoire de Math\'ematiques Jean Leray, UMR CNRS 6629, \\ \small 2 Rue de la Houssini\`ere BP 92208, F-44322 Nantes Cedex 03. \\
\small Email: francois.nicoleau@univ-nantes.fr }





\maketitle


\begin{abstract}
We show that there is non-uniqueness for the Calder\'on problem with partial data for Riemannian metrics with Hölder continuous coefficients in dimension greater or equal than three. We provide simple counterexamples in the case of cylindrical Riemannian manifolds with boundary having two ends. The coefficients of these metrics are smooth in the interior of the manifold and are only Hölder continuous of order $\rho <1$ at the end where the measurements are made. More precisely, we construct a toroidal ring $(M,g)$ which is not a warped product manifold, and we show that there exist in the conformal class of $g$ an infinite number of Riemannian metrics $\tilde{g} = c^4 g$ such that their corresponding partial Dirichlet-to-Neumann maps at one end coincide. The corresponding smooth conformal factors are harmonic with respect to the metric $g$ and do not satisfy the unique continuation principle.


\vspace{0.5cm}

\noindent \textit{Keywords}. Inverse problems, Anisotropic Calder\'on problem, Unique continuation principle.


\noindent \textit{2010 Mathematics Subject Classification}. Primaries 81U40, 35P25; Secondary 58J50.

\end{abstract}

\tableofcontents


\Section{Introduction} \label{0}

\subsection{The anisotropic Calder\'on problem with partial data}

The anisotropic Calder\'on problem on compact connected Riemannian manifolds with boundary is one of the most famous examples of inverse problems for an elliptic equation. The original problem that Calder\'on considered was whether one can recover the physical properties of a medium (like its electrical conductivity) by making only voltage and currents measurements at its boundary. This latter problem can be naturally formulated as a problem of geometric analysis in terms of the Dirichlet-to-Neumann map, (DN map in short), for the Laplace-Beltrami operator on Riemannian manifolds with boundary. We refer to the surveys \cite{GT2, KS2, Sa, U1} for a description of the current state of the art on the general anisotropic Calder\'on problem and also to \cite{ALP, AP, DSFKSU, DSFKLS, GSB, GT1, Is, IUY2, KKL, KKLM, KS1, KSU, LaTU, LaU, LeU} for important contributions to the question of uniqueness. On one hand, the uniqueness issue in the Calder\'on problem is still far from being completely understood in the case of smooth Riemannian manifolds of dimension greater or equal than $3$, and remains a major open problem. On the other hand, some counterexamples to uniqueness in the case in which the Dirichlet and Neumann data are measured on disjoint subsets of the boundary were found in our recent papers \cite{DKN2, DKN3, DKN4}.

\vspace{0.1cm}

The main goal of this paper is to give some non-uniqueness results for the anisotropic Calder\'on problem with partial data, (i.e in the case where the Dirichlet and Neumann measurements are made on the same open subset $\Gamma$ of the boundary), for a class of metrics whose coefficients are smooth in the interior of the manifold and H\"older continuous on the subset of the boundary where the measurements are made.

\vspace{0.1cm}
In order to state our main result, let us first recall the definition of the DN map (see for instance \cite{LeU} for the geometric formulation of the DN map for smooth Riemannian manifolds that we use here and \cite{AHG} for the formulation of the DN map corresponding to the original Calder\'on problem in terms of anisotropic conductivities with coefficients which are only $L^\infty$). Let $M$ be an $n$-dimensional smooth compact connected manifold with smooth boundary $\partial M$. We assume that this manifold $M$ is equipped with a Riemannian metric $g= (g_{ij})$ with measurable bounded coefficients satisfying (in local coordinates) the uniform ellipticity condition
\begin{equation}
\sum_{i,j} g^{ij}(x) \xi_i \xi_j \geq c |\xi|^2 \ {\rm{for \ a.e}} \  x \in M \ {\rm{and}} \ \xi \in \R^n,
\end{equation}
where the constant $c$ is strictly positive and where  $\left(g^{ij}\right)$ is the inverse of $(g_{ij})$.

On the Riemannian manifold $(M,g)$ we consider the Laplace-Beltrami operator $\Delta_{LB}$, given in local coordinates by
$$
\Delta_{LB}=  -\Delta_g = -\frac{1}{\sqrt{|g|}} \partial_i \left( \sqrt{|g|} g^{ij} \partial_j \right).
$$
where  $|g| = \det \left(g_{ij}\right)$ is the determinant of the metric tensor $(g_{ij})$, and where we use the Einstein summation convention.
It is well-known  that the Laplace-Beltrami operator $-\Delta_g$ with Dirichlet boundary conditions on $\partial M$ is self-adjoint on $L^2(M, dVol_g)$ and that $0$ is not an eigenvalue of $-\Delta_g$.

Let us consider the Dirichlet problem
\begin{equation} \label{Eq00}
  \left\{ \begin{array}{cc} -\Delta_g u = 0, & \textrm{on} \ M, \\ u = \psi, & \textrm{on} \ \partial M. \end{array} \right.
\end{equation}
A classical  result (see for instance \cite{AHG, GT, Sa, Ta1}) ensures that for any $\psi \in H^{1/2}(\partial M)$, there exists a unique weak solution $u \in H^1(M)$ of the Dirichlet problem (\ref{Eq00}). We recall that $u$ is a weak solution of (\ref{Eq00}) if
\begin{equation}\label{weaksolution}
\int_M \langle du, dw \rangle_g \ dVol_g =0 \ {\rm{for \ all}} \ w \in H_0^1(M),
\end{equation}
where $\langle du, dw \rangle_g$ is the pointwise scalar product of the one-forms $du,\, dv$ on $M$ induced by $g$ and given in local coordinates by
$\langle du, dw \rangle_g = g^{ij}\partial_i u \, \partial_j w$, and if the trace of $u$ on the boundary is equal to $\psi$. So, we can define the DN map as an operator $\Lambda_{g}$ from $H^{1/2}(\partial M)$ to $H^{-1/2}(\partial M)$ by
\begin{equation}
\left\langle \Lambda_{g} \psi | \phi \right \rangle = \int_M \langle du, dv \rangle_g \, dVol_g, \ {\rm{for \ all}} \ \psi, \ \phi \in H^{1/2}(\partial M),
\end{equation}
where $u$ is the unique weak solution of the Dirichlet problem (\ref{Eq00}), $v$ is any element of $H^1(M)$ such that $v_{|\partial M} = \phi$, and
$\left\langle \cdot  | \cdot  \right \rangle$ is the standard $L^2$ duality pairing between $ H^{1/2}(\partial M)$ and its dual.
In the case where the metric $g$ and the function $\psi$ are smooth, this definition coincides with the usual one
\begin{equation} \label{DN-Abstract}
\Lambda_{g} (\psi) = \left( \partial_\nu u \right)_{|\partial M},
\end{equation}
where  $\left( \partial_\nu u \right)_{|\partial M}$ is the normal derivative of $u$ with respect to the unit outer normal $\nu$ on $\partial M$.

\vspace{0.1cm}
As mentioned earlier, we are interested in the case in which the Dirichlet and Neumann data are measured on the same non-empty open subset $\Gamma$ of the boundary $\partial M$. Let us introduce the subspace of $H^{1/2}(\partial M)$ defined by:
\begin{equation}\label{sobolevlocal}
H_{co}^{1/2}(\Gamma) = \{ f \in H^{1/2} (\partial M) \ | \  supp f \subset \Gamma \}.
\end{equation}
The {\it{partial}} DN map is defined in a weak formulation as the operator $\Lambda_{g,\Gamma}$ such that
\begin{equation}\label{Partial-DNmap}
 \left\langle \Lambda_{g, \Gamma} (\psi) | \phi \right \rangle = \int_M \langle du, dv \rangle_g \, dVol_g, \quad \textrm{for all} \ \ \psi, \ \phi \in H_{co}^{1/2}(\Gamma),
\end{equation}
where $u$ is the unique weak solution of the Dirichlet problem (\ref{Eq00}), and where $v$ is any element of $H^1(M)$ such that $v_{|\partial M} = \phi$. As previously, for smooth metrics $g$ and smooth boundary data $\psi$, the partial DN map is simply given by:
\begin{equation}
  \Lambda_{g,\Gamma} (\psi) = \left( \partial_\nu u \right)_{|\Gamma}.
\end{equation}

\vspace{0.2cm}\noindent
In its simplest form, the anisotropic Calder\'on problem with partial data can be stated as follows:

\vspace{0.1cm}

\emph{If a pair of partial DN maps $\Lambda_{g_1,\Gamma}$ and $\Lambda_{g_2,\Gamma}$ coincide, is it true that the metrics $g_1$ and $g_2$ are the same}?

\vspace{0.1cm}\noindent
Because of several natural and geometric gauge invariances, the answer to the question stated above turns out to be negative. These lead to refined formulations of the Calder\'on problem that we shall present shortly, and that constitute the actual statement of this inverse problem. Indeed, it results from the definition (\ref{Eq00}) - (\ref{Partial-DNmap}) that the partial DN map $\Lambda_{g, \Gamma}$ is invariant when the metric $g$ is pulled back by any diffeomorphism of $M$ whose restriction to $\Gamma$ is the identity, \textit{i.e.}
\begin{equation} \label{Inv-Diff}
  \forall \phi \in \textrm{Diff}(M) \ \textrm{such that} \ \phi_{|\Gamma} = Id, \quad \Lambda_{\phi^*g, \Gamma} = \Lambda_{g, \Gamma}.
\end{equation}


In dimension two, there is another gauge invariance of the DN map due to the conformal invariance of the Laplace-Beltrami operator. More precisely, recall that in dimension $2$
$$
  \Delta_{cg} = \frac{1}{c} \Delta_g,
$$
for any smooth function $c >0$. Therefore, we have in dimension $2$
\begin{equation} \label{Inv-Conf}
  \forall c \in C^\infty(M), \ c >0 \ \textrm{and} \ c_{|\Gamma} = 1, \ \Lambda_{c g, \Gamma} = \Lambda_{g, \Gamma},
\end{equation}
since the unit outer normal vectors $\nu_{cg}$ and $\nu_g$ are identical on $\Gamma$.


\vspace{0.1cm}
It follows that the appropriate question (called the \emph{anisotropic Calder\'on conjecture with partial data}) to address is the following:
\vspace{0.1cm}

\emph{Let $M$ be a $n$-dimensional smooth compact connected manifold with smooth boundary $\partial M$. 
Let $g,\, \tilde{g}$ denote Riemannian metrics on $M$ with measurable bounded coefficients and let $\Gamma$ be an open subset of $\partial M$. Suppose that
$$
  \Lambda_{g,\Gamma} = \Lambda_{\tilde{g},\Gamma}.
$$
Does it follow that $g = \tilde{g} $  up to the gauge invariance (\ref{Inv-Diff}) if $n \geq 3$, and up to the gauge invariances (\ref{Inv-Diff})-(\ref{Inv-Conf}) in dimension $n=2$.
}\\

\vspace{0.1cm}

One may also consider a simpler inverse problem  by assuming that the Riemannian manifolds $(M,g)$ and $(M,\tilde{g})$ belong to the same conformal class, that is $\tilde{g} = c g$ 
for some positive smooth function c on $M$. In that case, the anisotropic Calder\'on problem amounts to the following statement \\

\emph{Let $M$ be a $n$-dimensional smooth compact connected manifold with smooth boundary $\partial M$. 
Let $g$ denotes a Riemannian metric on $M$ with measurable bounded coefficients and let $\Gamma$ be an open subset of $\partial M$. Let $c$ be a smooth positive function on $M$. If
$$
  \Lambda_{c g,\Gamma} = \Lambda_{g,\Gamma},
$$
then is it true that}
\begin{equation} \label{Inv-Conformal-1}
  c = 1, \quad \textrm{on} \ M?
\end{equation}
In fact, according to (\ref{Inv-Diff}), the assumption $\Lambda_{c g,\Gamma} = \Lambda_{g,\Gamma}$ should entail the question: does there exist a diffeomorphism 
$\phi: \, M \longrightarrow M$ with $\phi_{| \, \Gamma} = Id$ such that
\begin{equation} \label{Inv-Conformal}
  \phi^* g = c g?
\end{equation}
But, as was proved by Lionheart \cite{Li} {\it{for smooth metrics}}, any diffeomorphism $\phi: \, M \longrightarrow M$  which satisfies $\phi^* g = c g$ and $\phi_{|\Gamma} = Id$ 
for a non-empty open subset $\Gamma$ of $\partial M$ is the identity on the whole manifold $M$. Thus, the condition (\ref{Inv-Conformal}) may therefore be replaced by the condition (\ref{Inv-Conformal-1}).

\vspace{0.1cm}

Finally, there exists a last version of the anisotropic Calder\'on problem with partial data on $\Gamma$ involving an external potential. Consider the Dirichlet problem for the Schr\"odinger equation on $(M,g)$ with potential $V \in L^\infty(M)$
\begin{equation} \label{Eq0-Schrodinger}
  \left\{ \begin{array}{cc}
  (-\Delta_g + V) u = 0, & \textrm{on} \ M, \\ u = \psi, & \textrm{on} \ \partial M,
   \end{array} \right.
\end{equation}
where $\psi \in H_{co}^{1/2}(\Gamma)$. We assume that $0$ does not belong to the Dirichlet spectrum of $-\Delta_g +V$. Then, there exists a unique weak solution $u \in H^1(M)$ of (\ref{Eq0-Schrodinger}) (see for instance \cite{DSFKSU, Sa}). As previously, this allows us to define in the same way (\textit{i.e.} in a weak sense) the partial Dirichlet-to-Neumann map $\Lambda_{g, V, \Gamma} (\psi)$ for all $\psi \in H_{co}^{1/2}(\Gamma)$.

\vspace{0.1cm}

For \emph{smooth} Riemannian metrics $g$, it is well-known that there is a close connection between the anisotropic Calder\'on problem for Schr\"odinger operators and the anisotropic Calder\'on problem within the conformal class of a fixed metric $g$. It is based on the transformation law for the Laplace-Beltrami operator on a $n$-dimensional Riemannian manifold $(M,g)$ under conformal changes of the metric, that is
\begin{equation} \label{ConformalScaling}
  -\Delta_{c^4 g} u = c^{-(n+2)} \left( -\Delta_g + q_{g,c} \right) \left( c^{n-2} u \right),
\end{equation}
where the potential $q_{g,c}$ is given by
\begin{equation} \label{q}
  q_{g,c} = c^{-n+2} \Delta_{g} c^{n-2}.
\end{equation}

\vspace{0.1cm}\noindent
As a by product,  we get for instance the following result for \emph{smooth} metrics (see \cite{DKN3}, Proposition 1.1):

\begin{prop} \label{smoothgauge}
Let $\Gamma$ be any fixed open set of $\partial M$. Assume that $c$ is a smooth positive function on $M$ such that $c = 1$ on $\Gamma$ and $\partial_{\nu} c =0$ on $\Gamma$. Then
\begin{equation} \label{Link}
	\Lambda_{c^4 g, \Gamma} = \Lambda_{g, q_{g,c}, \Gamma}.
\end{equation}
In particular, if the potential $q_{g,c}=0$, \textit{i.e.} if the conformal factor $c$ satisfies additionally
\begin{equation} \label{Eq-Conf}
  \Delta_g c^{n-2} = 0, \ \textrm{on} \ M,
\end{equation}
we get immediately
$$
  \Lambda_{c^4 g, \Gamma} = \Lambda_{g, \Gamma}\,.
$$
\end{prop}

\vspace{0.1cm}
Note that in dimension $n=2$, (\ref{Eq-Conf}) is automatically satisfied and (\ref{Link}) corresponds simply to the gauge invariance (\ref{Inv-Conformal}). In dimension $n \geq 3$ and for smooth metrics $g$, the unique continuation principle implies that $c=1$ on $M$ (remember that not only (\ref{Eq-Conf}) is satisfied, but $c$ must be identically $1$ on $\Gamma$ and $\partial_\nu c$ must be $0$ on $\Gamma$). Thus, if we want to use the result in Proposition \ref{smoothgauge} to obtain a counterexample to uniqueness for the anisotropic Calder\'on problem with partial data, we need to construct a metric $g$ in such a way the Laplace-Beltrami operator $-\Delta_g$ \emph{does not satisfy the unique continuation principle}.

\vspace{0.1cm}
We recall that, in dimension $n \geq 3$, the unique continuation principle holds for a uniformly elliptic operator on a domain $\Omega$ if the coefficients of the principal part of this operator are locally Lipschitz continuous, whereas in dimension $n=2$, the unique continuation principle holds if the coefficients of the principal part are $L^{\infty}$ (see for instance \cite{Hor1, Hor2, Tat}). Nevertheless, in dimension $n =3$, if the coefficients of the principal part are only \emph{Hölder continuous}, there exist examples of \emph{nonunique continuation}. The first such example was given in 1963 by Pli\'{s} \cite{Pli}, and later in 1972, a sharper counterexample was found by Miller \cite{Mil} for an elliptic equation in divergence form. This divergence form is very well adapted with our Riemannian setting. So, the main and basic idea of our paper is to construct a metric $g$ on a suitable manifold $M$ such that the Laplacian $\Delta_g$ is nothing but Miller's elliptic operator and the conformal factor $c$ is very close to Miller's solution.

\vspace{0.1cm}
But, before giving this construction, our first task is to slightly extend Proposition \ref{smoothgauge} for metrics $g$ having coefficients in $L^{\infty}(M)$ since, in this case, the potential $q_{g,c}$ only has a distributional sense. In other words, we have to write Proposition \ref{smoothgauge} (with $q_{g,c}=0$) in a weak sense. To do this, we remark that, for a smooth metric $g$ and for a smooth conformal factor $c$, the conditions $\partial_{\nu} c =0$ on $\Gamma$ and $\Delta_{g} c^{n-2}=0$ on $M$ are equivalent to
\begin{equation}\label{weakconformal}
 \int_M \langle d (c^{n-2}), d w \rangle_g \ dVol_g = 0 \ ,\ \forall w \in H^1 (M) \ {\rm{such \ that}} \ supp \ w_{|\partial M} \subset \Gamma,
\end{equation}
thanks to the Green's formula:
\begin{equation}\label{green}
\int_M \Delta_g (c^{n-2}) w  \ dVol_g + \int_M \langle d (c^{n-2}), d w \rangle_g \ dVol_g = \int_{\partial M} \partial_{\nu} c^{n-2} w \ d{\sigma}_{g}\,,
\end{equation}
where $dVol_g $ denotes the Riemannian volume element and $d{\sigma}_{g}$ denotes the volume element induced by $g$ on $\partial M$.
\vspace{0.2cm}

Now, we can state the following extension of Proposition \ref{smoothgauge} for {\it{metrics with bounded measurable coefficients}} which is one of the main arguments of our counterexamples for the anisotropic Calder\'on problem with partial data.

\begin{prop} \label{weakgauge}
Let $\Gamma$ be any fixed open set of $\partial M$. Assume that $c$ is a smooth positive function on $M$ such that $c = 1$ on $\Gamma$ and such that
\begin{equation}\label{weakconformal0}
 \int_M \langle d (c^{n-2}), d w \rangle_g \ dVol_g = 0 \ ,\ \forall w \in H^1 (M) \ {\rm{such \ that}} \ supp \ w_{|\partial M} \subset \Gamma.
\end{equation}
Then,
\begin{equation} \label{Link1}
	\Lambda_{c^4 g, \Gamma} = \Lambda_{g, \Gamma}.
\end{equation}
\end{prop}

\begin{proof}
 For any $\psi, \ \phi \in C_0^{\infty}(\Gamma)$, the partial DN map $\Lambda_{c^4g, \Gamma}$  is given by the relation:
 \begin{equation}\label{weakdn}
  \langle \Lambda_{c^4g, \Gamma} \psi \, | \, \phi > = \int_M \ \langle du, dw \rangle_{c^4g} \ dVol_{c^4 g},
 \end{equation}
where $u \in H^1 (M)$ is the unique weak solution of (\ref{Eq00}) associated to the metric $c^4 g$  with $u_{|\partial M} = \psi$ and $w \in C^{\infty}(M)$ is any extension of $\phi$. Note that the existence of such an extension is given for instance in the monograph \cite{Lee}, Corollary 6.27, together with the Tietze's extension theorem. As a consequence, the function $c^{n-2}uw \in H¹(M)$ and its trace on the boundary has its support in $\Gamma$. Now, a straightforward algebraic calculation gives:
\begin{eqnarray}\label{wconformalLaplacian}
 \int_M \ \langle du, dw \rangle_{c^4g} \ dVol_{c^4 g} &=&  \int_M \ \langle d(c^{n-2}u), d(c^{n-2}w) \rangle_{g} \ dVol_{g}  \nonumber \\
                                          & & -\int_M \ \langle d(c^{n-2}), d(c^{n-2}uw) \rangle_{g} \ dVol_{g}.
\end{eqnarray}
Thus thanks to the hypothesis (\ref{weakconformal}), we get:
\begin{equation}\label{eqinter}
 \langle \Lambda_{c^4g, \Gamma} \psi \, | \, \phi \rangle =  \int_M \ \langle d(c^{n-2}u), d(c^{n-2}w) \rangle_{g} \ dVol_{g}.
\end{equation}
Let us now prove that $v = c^{n-2} u$ is a weak solution of (\ref{Eq00}). Indeed, since $u$ is the unique weak solution of (\ref{Eq00}) associated to the metric $c^4 g$, we have for any $\varphi \in C_0^{\infty} (M)$,
\begin{equation}
\int_M \ \langle du, d\varphi \rangle_{c^4g} \ dVol_{c^4 g} =0.
\end{equation}
Using again the relation (\ref{wconformalLaplacian}) (with $w$ replaced by $\varphi$) and (\ref{weakconformal}), we get for any $\varphi \in C_0^{\infty} (M)$,
\begin{equation}
 \int_M \ \langle d(c^{n-2}u), d(c^{n-2}\varphi) \rangle_{g} \ dVol_{g} =0
\end{equation}
It follows that $v= c^{n-2}u$ is a weak solution of (\ref{Eq00}) for the metric $g$ which satisfies $v_{|\partial M} = \psi$ since $c =1$ on $\Gamma$.

Then, using the definition of the partial DN map $\Lambda_{g, \Gamma}$ and (\ref{eqinter}) again, we get immediately
\begin{equation}
 \langle \Lambda_{c^4g, \Gamma} \psi \, | \, \phi \rangle = \langle \Lambda_{g, \Gamma} \psi \, | \, \phi \rangle \ {\rm{for \ all}} \ \psi, \ \phi \in C_0^{\infty}(\Gamma).
\end{equation}
We conclude the proof by a standard density argument.

\end{proof}

\subsection{Statement of the main result}

Let us introduce first some notations. We consider the $n$-dimensional manifold
$$
  M= [0,1] \times T^{n-1},
$$		
where $T^{n-1}$ denotes the $(n-1)$-dimensional torus, ($n \geq 3$). This manifold has the topology of a cylinder. Using the standard toroidal coordinates, we 
can also interpret $M$ as a toroidal ring (see \cite{Gia}, Remark 2.5). Note that the boundary of $M$ is disconnected and consists in the disjoint union of 
two copies of $T^{n-1}$, (which we call {\it{ends}} in this paper), more precisely
$$
  \partial M = \Gamma_0 \cup \Gamma_1, \quad \Gamma_0 = \{0\} \times T^{n-1}, \quad \Gamma_1 = \{1\} \times T^{n-1}.
$$

\vspace{0.1cm}\noindent
Our main result is the following:

\begin{thm} \label{Main}
There exists a Riemannian metric $g$ on $M$ whose coefficients are smooth in $[0,1) \times T^{n-1}$ and Hölder continuous of order $\rho<1$ on the end $\Gamma_1$, 
and there exist an infinite number of smooth positive conformal factors $c$ which are not identical to $1$ on $M$, such that the following partial DN maps on $\Gamma_1$ are identical:
\begin{equation}
\Lambda_{c^4 g, \Gamma_1} = \Lambda_{g, \Gamma_1}.
\end{equation}
\end{thm}

As we have said earlier, the proof of this theorem is rather simple and relies on Miller's famous counterexample to unique continuation for a uniformly elliptic equation in divergence form in dimension $3$ (see \cite{Mil}). Assume for a moment that the dimension of our manifold $M$ is $3$ and let us summarize the strategy of the proof. We consider first a metric $g$ on $M$ such that $\Delta_g$ is precisely the uniformly elliptic operator from Miller's construction. Note in passing that the elliptic operator from \cite{Mil} naturally \emph{lives} on the cylinder $M = [0,1] \times T^2$ as was noticed by Gianotti in \cite{Gia}. Then, using Miller's solution of this elliptic PDE, we shall construct in section 2 an {\it{infinite family}} of smooth conformal factors $c$ satisfying the assumptions of Proposition \ref{weakgauge}, i.e  $c = 1$ on $\Gamma_1$, and $\Delta_g c^{n-2}=0$, $\partial_{\nu} c=0$ on $\Gamma_1$ in the following weak sense:
\begin{equation}\label{localnature}
 \int_M \langle d (c^{n-2}), d w \rangle_g \ dVol_g = 0 \ ,\ \forall w \in H^1 (M) \ {\rm{such \ that}} \ supp \ w_{|\partial M} \subset \Gamma_1.
\end{equation}
This leads automatically to counterexamples to uniqueness for the Calder\'on problem with partial data in dimension $3$ since the metrics $g$ and $c^4g$ are not isometric (see the proof of Theorem \ref{Main}).
The proof in the case of higher dimensions is similar.

\begin{rem}
  It is important to stress that, even though two of the coefficients (namely $A_1 (t)$ and $A_3 (t)$) of the elliptic operator $\Delta_g$ are only Hölder continuous functions, $\Delta_g c^{n-2}$ is classically well-defined since these two functions are not differentiated with respect to $t$ in the expression of $\Delta_{g}$, (see \cite{Mil}, Theorem 1, or Proposition \ref{MillerCE} in this paper). In other words, the equation $\Delta_g c^{n-2}=0$ on $M$ can be also  understood in the {\it{classical sense}}.
\end{rem}
\begin{rem}
  All the derivatives of the conformal factors $c$ at the end $\Gamma_1$ are equal to zero as one would expect from boundary determination results (see \cite{KY}).
\end{rem}


\begin{rem}
We emphasize that this theorem is of a local nature. We cannot obtain with the same approach a counterexample for the anisotropic Calder\'on problem with global data, i.e when $\Gamma = \partial M$. Indeed, if the smooth conformal factor satisfies $c = 1$ on $\partial M$ and
\begin{equation}\label{localnature1}
 \int_M < d (c^{n-2}), d w >_g \ dVol_g = 0 \ ,\ \forall w \in H^1 (M),
\end{equation}
then, choosing $w = c^{n-2}$ in (\ref{localnature1}), we obtain immediately that $c$ is identical to $1$ on $M$. An alternative interpretation is to say that $0$ is not a Dirichlet eigenvalue of the Laplace-Beltrami operator $\Delta_g$.
\end{rem}

%
%

\subsection{A brief history of known results on the anisotropic Calder\'on problem}

In this last part of the Introduction, we give a brief and non-exhaustive survey of some of the most important known results on the anisotropic Calder\'on conjecture.

In dimension $2$, the anisotropic Calder\'on conjecture for global and partial data has been been settled positively for compact connected Riemannian surfaces in \cite{LaU, LeU}. We also refer to \cite{ALP, AP} for similar results for global and partial data in the case of anisotropic conductivities which are only $L^\infty$ on bounded domains of $\R^2$. In dimension $n \geq 3$, if the Riemannian manifold is real analytic, compact and connected, with real analytic boundary, a positive answer for global (i.e when $\Gamma = \partial M$), and partial data has been given in \cite{LeU,LaU, LaTU}. Similarly, the global anisotropic Calder\'on problem has been answered positively for compact connected Einstein manifolds with boundary in \cite{GSB}.

If the background metrics are not assumed to be analytic, the general anisotropic Calder\'on problem  in dimension $n\geq 3$  is still a difficult open problem, whether one is dealing with the case of full or partial data. Nevertheless, some important results have recently appeared in \cite{DSFKSU, DSFKLS, KS1}, for special classes of smooth compact connected admissible Riemannian manifolds with boundary. By definition, admissible manifolds $(M,g)$ are \emph{conformally transversally anisotropic},
$$
  M \subset \subset \R \times M_0, \quad g = c ( e \oplus g_0),
$$
where $(M_0,g_0)$ is an $n-1$ dimensional smooth compact connected Riemannian manifold with boundary, $e$ is the Euclidean metric on the real line and $c$ is a smooth strictly positive function in the cylinder $\R \times M_0$.  It has been shown in \cite{DSFKSU, DSFKLS} that for admissible manifolds and under some geometric assumptions on the transverse manifolds $M_0$ (see for instance \cite{DSFKLS} for a precise statement), the conformal factor $c$ is uniquely determined from the knowledge of the DN map. These results have been further extended to the case of partial data in \cite{KS1}. We also refer to \cite{GT1, Is, IUY1} for additional results in the  case of local data and to the surveys \cite{GT2, KS2} for further references.

\vspace{0.1cm}
Let us also  mention several papers dealing with the Calder\'on problem for \emph{more singular} metrics or conductivities in dimension $n \geq 3$. Haberman and Tataru \cite{HaTa} showed uniqueness in the global Calder\'on problem for uniformly elliptic isotropic conductivities that are Lipschitz and close to the identity. The latter condition was relaxed by Caro and Rogers in \cite{CaRo}. In dimensions $3$ and $4$, these results were slightly improved by Haberman in \cite{Ha} to the case of conductivities that belong to $W^{1,n}$. Related to the partial Calder\'on problem, Krupchyk and Uhlmann in \cite{KrUh1} proved that an isotropic conductivity with - roughly speaking - $\frac{3}{2}$ derivatives in the $L^2$ sense is uniquely determined by a DN map measured on possibly very small subset of the boundary.

\vspace{0.1cm}
Finally, we conclude this introduction mentioning the series of papers on \emph{invisibility} by Greenleaf, Kurylev, Lassas and Uhlmann (see \cite{GLU} for the original paper, \cite{ALP2} for an extension of this work in dimension $2$ and \cite{GKLU2, U2} for thorough surveys in this field). In these works, some counterexamples to uniqueness to the \emph{global} Calder\'on problem were described. These counterexamples are obtained for a class of metrics that are \emph{highly singular} at a given closed hypersurface lying within the manifold in the sense that the metric degenerates or blows up at this hypersurface. This is a situation in sharp contrast to the one obtained in the present work in which the metrics remain \emph{positive definite} and \emph{H\"older continuous} everywhere in the manifold, (and even smooth in the interior). Coming back to invisibility, we also refer to our recent paper \cite{DKN5} where similar non-uniqueness results were obtained for singular warped product metrics on the same class of manifolds $M$ as the ones used in this paper.

%
%

\section{Counterexamples to uniqueness} \label{3D}


\subsection{Miller's counterexample to unique continuation principle}

In this section, we recall the remarkable counterexamples obtained by Miller \cite{Mil} to the unique continuation principle, counterexamples that were slightly improved later by Mandache \cite{Man}. We say that a partial differential equation $P(x,D)u=0$ on a domain $\Omega$ possesses the {\it{unique continuation property}} if the equality $u=0$ in some ball within $\Omega$ implies the equality $u=0$ on $\Omega$.

\vspace{0.1cm}
\par
In dimension $n \geq 3$, the unique continuation property holds for a uniformly elliptic operator on a domain $\Omega$ if the coefficients of the principal part of this operator are locally Lipschitz continuous (see for instance \cite{Hor1, Hor2, Tat}). If the coefficients are only Hölder continuous, then there exist examples of \emph{nonunique continuation}. The first one was given in 1963 by Pli\'{s} \cite{Pli}. He considered a uniformly elliptic equation on a domain of $\R^3$ having the form:
\begin{equation}\label{Plis}
\sum_{i,j=1}^3 a_{ij}(x) \ \partial_{ij}^2 u + \sum_{i=1}^3 b_i (x)\  \partial_i u+ c(x) u =0,
\end{equation}
where the coefficients of this equation are Hölder continuous with order less than $1$. The coefficient of zero-order term $c(x)$  has no constant sign and might explain the existence of this counterexample. In 1972, a sharper counterexample (without zero-order term) was found by Miller \cite{Mil}. He constructed a smooth solution $u(t,x,y)$ of a uniformly elliptic equation in divergence form:
\begin{equation}\label{divergence}
div \  ( \mathcal{A} \ \nabla u) =0,
\end{equation}
where the $(3 \times 3)$ symmetric matrix $\mathcal{A}$ is given by
\begin{equation}\label{matricemiller}
\mathcal{A}=\left(
\begin{array}{c c c}
1 &0&0\\ 0&1+a_1 (t,x,y)+A_1 (t)&a_2(t,x,y)\\ 0&a_2(t,x,y)&1+a_3(t,x,y)+A_3(t)
\end{array} \right) .
\end{equation}
This  matrix $\mathcal{A}$ has its eigenvalues in $[\alpha, \alpha^{-1}]$ with elliptic constant $\alpha \in ]0,1[$. More precisely, Miller proved the following result:

\begin{thm}[Miller \cite{Mil}] \label{MillerCE}
There exists an example of nonunique continuation on the half-space $E= [0, +\infty[ \times \R^2$ for a uniformly elliptic equation
\begin{equation}\label{Millereq}
\partial_t^2 u + \partial_x((1+a_1+A_1) \partial_x u ) + \partial_x (a_2 \partial_y u) + \partial_y (a_2 \partial_x u) + \partial_y((1+a_3+A_3) \partial_y u )=0 \ {\rm{in}} \ E.
\end{equation}
\begin{enumerate}
\item The classical solution $u(t,x,y)$ is $C^{\infty}$ on $E$, identically zero for $t \geq T>0$, but not identically zero in any open subset of $[0,T[\times \R^2$.
\item The coefficients $a_1(t,x,y), \ a_2(t,x,y), \ a_3(t,x,y)$ are $C^{\infty}$ on $E$ and are identically zero for $t \geq T$.
\item The coefficients $A_1(t), \ A_3(t)$ are Hölder continuous of order $\frac{1}{6}$ on $[0,\infty[$, $C^{\infty}$ on $[0,T[$, and identically zero for $t \geq T$.
\item All functions $u, \ a_1, \ a_2, \ a_3$ are periodic in $x$ and $y$ with period $2\pi$.
\item Although the coefficient matrix $\mathcal{A}$ is only Hölder continuous at $t=T$, $u$ is a classical (as well as weak) solution of (\ref{Millereq}) on $E$.
\end{enumerate}
\end{thm}

We emphasize that this theorem can be improved as follows: modifying slightly Miller's initial proof, the coefficients $A_1(t), \ A_3(t)$ can be actually constructed in such a 
way that they are Hölder continuous functions of fixed order $\rho <1$ , (see \cite{Mil}, Remarks p. 115), as it occurs in Mandache's paper \cite{Man}. But, we prefer to 
use Miller's work rather Mandache's paper \cite{Man} since it is not clear for us that the coefficients of the elliptic operator constructed by Mandache are smooth on $[0,T[$. 
It is also important to say again that this function $u$ is a classical solution of the elliptic equation (\ref{Millereq}) since we do not differentiate the Hölder functions 
$A_1(t)$ and $A_3(t)$ with respect to $t$ in the elliptic equation (\ref{Millereq}).

\vspace{0.1cm}

Moreover, since the solution $u(t,x,y)$ is periodic in $(x,y)$ with period $2\pi$, as was observed by C. Giannotti in \cite{Gia}, 
Miller's solution can be considered as a solution to an elliptic equation on the toroidal ring $M= [0,T]\times T^2$ where $T^2$ is the usual $2$-dimensional torus.
\vspace{0.1cm}

Note that, in the following section, and in order to simplify the notation, we assume (without loss of generality) that $T=1$.


\subsection{Construction of the counterexamples on a toroidal ring of $\R^3$} \label{Model-DNmap}

In this section, we consider a Riemannian manifold $(M,g)$, which has the topology of a cylinder $M = [0,1] \times T^2$. We denote
$(t,x,y)$  a global coordinate system on $M$. The manifold $M$ can be interpreted as a toroidal ring (see \cite{Gia}, Remark 2.5). We equip this manifold with the following Riemannian metric:

\begin{equation} \label{Metric}
  g = D dt^2+(1+a_3+A_3)dx^2-2a_2dxdy+(1+a_1+A_1)dy^2,
\end{equation}
where the coefficients $a_1,a_2,a_3,A_1,A_3$ are given by Theorem \ref{MillerCE}, and
\begin{equation}\label{determinant}
D = det \ \mathcal{A}= (1+A_1(t)+a_1(t,x,y)) (1+A_3(t)+a_3(t,x,y)) -a_2^2(t,x,y).
\end{equation}
Clearly,  $(M,g)$ is not a warped-product manifold and we have $\sqrt{|g|} \ (g^{-1}) = \mathcal{A}$. Note that this metric is well-defined on $M$ since all the coefficients are periodic in the 
variables $x,y$ with period $2\pi$. We recall that the boundary $\partial M$ of $M$ is disconnected and consists in the disjoint union of two copies of $T^2$,  precisely
$$
  \partial M = \Gamma_0 \cup \Gamma_1, \quad \Gamma_0 = \{0\} \times T^2, \quad \Gamma_1 = \{1\} \times T^2.
$$
We emphasize that our metric $g$ is smooth everywhere inside the manifold (precisely on $\bar{M} \setminus \Gamma_1$) and Hölder continuous of order $\rho<1$ on the end $\Gamma_1$. Thanks to Theorem \ref{MillerCE} and as it was observed before, Miller's solution is  a {\it{classical harmonic function}} for the Laplace Beltrami operator $\Delta_g$, i.e it satisfies the Laplace equation in the classical sense
\begin{equation}\label{harmonic}
\Delta_g u = 0 \ {\rm{in}} \ M.
\end{equation}
Moreover, since the solution $u(t,x,y)$ is smooth on $E$ and is vanishing for $t \geq 1$, all the derivatives of $u$ are also equal to zero at $t=1$. In particular, one has:
\begin{equation}\label{trace}
u_{|\Gamma_1} = 0 \ ,\ \partial_{\nu} u_{|\Gamma_1} = 0,
\end{equation}
where $\partial_{\nu}$ is the normal derivative on $\Gamma_1$ with respect to the unit outer normal vector $\nu$ on $\partial M$.

\vspace{0.5cm}\noindent
Now, let us define our infinite family of conformal factors $c(t,x,y)$. We set
\begin{equation}\label{conffactor}
c_{\epsilon}(t,x,y) =1 + \epsilon \ u(t,x,y),
\end{equation}
where $\epsilon >0$ is chosen sufficiently small to ensure that $c_{\epsilon}(t,x,y) \geq \half$ on the compact manifold $M$.

\begin{proof}[Proof of Theorem \ref{Main} in dimension $3$]
The conformal factors $c_{\epsilon}$ are smooth on $M$. Moreover, thanks to Theorem \ref{MillerCE}, they are not identically $1$ on $M$. Using (\ref{harmonic}) and (\ref{trace}), these conformal factors $c$ satisfy obviously and in the classical sense:
\begin{equation}\label{traceconf}
\Delta_g c_{\epsilon} = 0 \ {\rm{in}} \ M \ ,\ c_{\epsilon |\Gamma_1} = 1 \ ,\ \partial_{\nu} c_{\epsilon|\Gamma_1} = 0.
\end{equation}
Then, thanks to Green's formula, these conformal factors $c$ satisfy the assumptions of Proposition \ref{weakgauge}, and one obtains thus immediately:
\begin{equation}
\Lambda_{c_{\epsilon}^4 g, \Gamma_1} = \Lambda_{g, \Gamma_1}.
\end{equation}
Now assume that, for all $\epsilon>0$,  there exists a diffeomorphism (depending on $\epsilon$) $\phi: \, M \longrightarrow M$ such that $\phi_{|\Gamma_1} = Id$ and $\phi^* g = c_{\epsilon}^4g$. In particular, since $\phi$ is a diffeomorphism, we see that $Vol_g(M) = Vol_{\phi^*g}(M) = Vol_{c_{\epsilon}^4g}(M)$. Hence we must have
\begin{equation}\label{asympt}
  \int_M [ (1+\epsilon u)^6 -1]\ \sqrt{|g|} \ dx  = 0 \ {\rm{for \ all\ }} \epsilon >0.
\end{equation}
The term of order $2$ of this polynomial in the variable $\epsilon$ must be equal to $0$, i.e
\begin{equation}
\int_M u^2 \sqrt{|g|} \ dx  = 0,
\end{equation}
which is not possible since $u$ is not identically equal to zero. Thus, there exists an infinite number of conformal factors $c_{\epsilon}$ such that the metrics $g$ and $c_{\epsilon}^4 g$ are not isometric 
and we see that we have found counterexamples to uniqueness for the partial anisotropic Calder\'on problem in dimension $3$.
\end{proof}

\subsection{Generalization in the $n$-dimensional case} \label{nD}

The above construction can be generalized in a straightforward way to the $n$-dimensional case with $n \geq 3$. Indeed, let us consider the manifold $M= [0,1] \times T^{n-1}$ where $T^{n-1}$ stands for the $(n-1)$ dimensional torus. We denote $(t,x_1,x_2, ..., x_{n-1})$  a global coordinate system on $M$. As in the previous section, we introduce the coefficient
\begin{equation}\label{determinant1}
D(t, x_1, x_2) = (1+A_1(t)+a_1(t,x_1,x_2)) (1+A_3(t)+a_3(t,x_1,x_2)) -a_2^2(t,x_1,x_2),
\end{equation}
where the coefficients $a_1,a_2,a_3,A_1,A_3$ are still given by Theorem \ref{MillerCE} with the identifications $x_1 = x$ and $x_2 = y$. Now, we equip this manifold $M$ with the Riemannian metric
\begin{equation} \label{Metricn}
  g = D^{\frac{1}{n-2}} \ \left( dt^2+ D^{-1} \left( (1+A_3+a_3)dx_1^2 -2a_2 dx_1 dx_2+(1+A_1+a_1)dx_2^2 \right)  + \sum_{i=3}^{n-1} dx_i^2 \right).
\end{equation}
As before, the metric $g$ is smooth inside the manifold $M$ and only H\"older continuous at the end $\Gamma_1 = \{1\} \times T^{n-1}$. Then, we define the conformal 
factors (which will not depend on the variables $x_i$ for $i\geq 3$) by
\begin{equation}\label{conffactor1}
c_{\epsilon}(t,x_1, ...,x_{n-1}) =\left( 1 + \epsilon \ u(t,x_1,x_2)\right)^{\frac{1}{n-2}},
\end{equation}
where $u(t,x_1,x_2)$ is Miller's solution given in Theorem \ref{MillerCE} and $\epsilon >0$ is small enough to ensure that $c_{\epsilon} > 0$ on $M$. Using the same arguments as in the previous section, one has:
\begin{equation}
\Lambda_{c_{\epsilon}^4 g, \Gamma_1} = \Lambda_{g, \Gamma_1},
\end{equation}
which implies that there is non-uniqueness for the partial anisotropic Calder\'on problem in dimension $n \geq 3$.

\begin{rem}
In the previous non-uniqueness results for the anisotropic Calder\'on problem with partial data, we considered smooth compact connected cylindrical manifolds equipped with a  metric whose coefficients are only Hölder continuous, and having two ends. If we remove the assumption of smoothness for the manifold, then we can allow a connected boundary for $M$ and obtain probably counterexamples to uniqueness in the partial Calder\'on problem.

More precisely, let us consider the product manifold $M = [0,1] \times \overline{\Omega}$ where $\Omega$ is any connected bounded open set of $\R^{n-1}$ with smooth boundary. Note that the boundary of $M$ is now connected and given by:
$$
  \partial M = \left( \{0\} \times \overline{\Omega} \right) \cup \left( \{1\} \times \overline{\Omega} \right) \cup \left( (0,1) \times \partial \Omega \right).
$$
Clearly, we lose the smoothness of the manifold since $M$ has corners. Nevertheless, if the Dirichlet and the Neumann data are measured on  $\Gamma =\{1\} \times \overline{\Omega}$, one can probably use the previous constructions to get counterexamples in this new setting.
\end{rem}



\end{document}